\documentclass[a4paper,12pt%
]{amsart}
\pdfoutput=1
\usepackage{amsmath, amssymb, amsthm, amscd}
\usepackage{graphicx}
\usepackage{enumerate}
\usepackage[all,cmtip]{xy}

\newcounter{theorem}
\newtheorem{thm}[theorem]{Theorem}
\newtheorem{lemma}[theorem]{Lemma}
\newtheorem{prop}[theorem]{Proposition}

\newtheorem{defn}[theorem]{Definition}

\theoremstyle{remark}
\newtheorem*{remark*}{Remark}

\newtheorem{example}[theorem]{Example}

\numberwithin{equation}{section}
\numberwithin{theorem}{section}

\newcommand{\defemph}{\emph}

\newcommand{\Q}{\mathbb{Q}}
\newcommand{\Z}{\mathbb{Z}}
\newcommand{\C}{\mathbb{C}}
\newcommand{\N}{\mathbb{N}}

\title[$p$-adic functionals]{$p$-adic functionals on finite-rank torsion-free abelian groups}

\author{Gregory R. Maloney}
\address{Newcastle University}
\subjclass[2010]{Primary: 20K15, %Torsion-free groups, finite rank
20K45 %Topological methods
Secondary: 20E18 %Limits, profinite groups
}
\keywords{Torsion-free abelian group, stationary, $p$-adic, functional}
\date{\today}

\begin{document}

\begin{abstract}
Let $p$ be a prime and $G$ be a torsion-free abelian group.  
A homomorphism from $G$ to the $p$-adic integers is called a $p$-adic functional on $G$.  
If $G$ has finite rank, then $G$ can be represented as an inductive limit of an inductive sequence of free abelian groups of the same rank, and the group of all $p$-adic functionals on $G$ is described in terms of this inductive sequence.  
If this inductive sequence is stationary---i.e., if the homomorphism is the same at every stage---then the group of $p$-adic functionals is described in particularly simple terms, as a right-submodule that is invariant under the module homomorphism that this group homomorphism induces.  
It is shown that the class consisting of all such stationary inductive limits is closed under quasi-isomorphism; this strengthens a previous classification result of Dugas.  
\end{abstract}

\maketitle

\section{Introduction}\label{SEC:intro}

Let $p$ be a prime and let $G$ be a torsion-free abelian group.  
The group of \defemph{$p$-adic functionals} on $G$ is the group $Hom (G,\Z_p)$, denoted here by $G^{*p}$; these functionals were introduced in \cite{M:functionals}, where they were studied using the techniques of functional analysis.  
The first main result of that work is that $G^{*p*p}$, the $p$-adic double dual of $G$, is the same as the pro-$p$ completion of $G$.  
The second main result is a classification result.  
This classification uses a sequence of matrices, indexed by prime numbers $p$, such that the $p$th matrix has entries in $\Z_p$.  
These matrices are closely related to the matrices of Kurosch \cite{K:matrices}, but, in general, are smaller, although they still contain enough information to be a complete isomorphism invariant of $G$.  

The goal of this work is to study $G^{*p}$ in the case that $G$ has finite rank (here the \defemph{rank} of a torsion-free abelian group refers to its torsion-free rank, that is, the size of any maximal linearly independent subset).  
In this case, $G$ is isomorphic to the inductive limit of an inductive sequence of free abelian groups, all with the same rank, and $G^{*p}$ has a simple description in terms of the maps in this inductive sequence.  
This is the content of Section \ref{SEC:inductive-limit}.  

Section \ref{SEC:tensor} explains the relationship between the $p$-adic double dual, $G^{*p*p}$, and the tensor product $G\otimes \Z_p$, if $G$ has finite rank.  
In particular, $G^{*p*p}$ is isomorphic to the quotient of $G\otimes \Z_p$ by its divisible subgroup (in a way that preserves the image of $G$ in these groups).  
This clarifies the relationship between the classification result of \cite{M:functionals} and the Kurosch matrices \cite{K:matrices}, because the Kurosch matrices express the elements of a maximal independent subset of $G$ in terms of a $\Q_p$-basis for the divisible subgroup of $G\otimes \Z_p$ and a $\Z_p$-basis for the quotient.  

Section \ref{SEC:stationary-states} deals with the special case where $G$ is \defemph{stationary}, that is, $G$ is isomorphic to an inductive limit of an inductive sequence of finite-rank free abelian groups in which the homomorphisms are the same at each stage.  
In this case, the group of $p$-adic functionals has a particularly simple description, which is expressible in terms of basic linear algebra applied to the homomorphism from the inductive sequence. 
This is the content of Theorem \ref{THM:stationary-states}, which is the main result of Section \ref{SEC:stationary-states}.  

Section \ref{SEC:isomorphism} deals with quasi-isomorphism of torsion-free abelian groups.  
The main result is Theorem \ref{THM:quasi-isomorphism}, which says that the class of stationary torsion-free abelian groups is closed under quasi-isomorphism.  
This makes it possible to strengthen a previous result of Dugas \cite{D:stationary}, who gave an explicit description of what a finite-rank torsion-free abelian group $G$ must look like, up to quasi-isomorphism (Definition \ref{DEF:quasi-isomorphic}), if it is quasi-isomorphic to a stationary group.  
Theorem \ref{THM:quasi-isomorphism} implies that this description applies to $G$ if and only if $G$ itself is stationary.  

\subsection{Notation}\label{SUBSEC:notation}

Let $p$ be a prime.  
$\Z_p$ denotes the ring of $p$-adic integers and $\Q_p$ denotes its field of fractions.  
The $p$-adic absolute value of an element $\alpha$ of $\Q_p$ or $\Z_p$ is denoted by $|\alpha |_p$.  
If $n\in\N$ and $R$ is a ring, then $R^n$ denotes the module consisting of $n\times 1$ columns with entries in $R$, while $(R^n)^*$ denotes its dual consisting of $1\times n$ rows.  
If $m,n\in\N$, then $M_{m\times n}(R)$ denotes the module of $m\times n$ matrices over $R$.  
For $r\in R^n$, $r^t$ denotes its transpose in $(R^n)^*$.  
For $A\in M_{m\times n}(R)$, $A^t$ denotes its transpose in $M_{n\times m}(R)$.  
For $v = (v_1,\ldots, v_n)^t \in \Z_p^n$, the norm $\| v\|_p $ is defined by $\| v\|_p := \max_{1\leq i\leq n} \{ \| v_i\|_p\}$.

\section{The state space of an inductive limit}\label{SEC:inductive-limit}

It is well known that every rank-$r$ torsion-free abelian group can be written as an increasing union of a chain of subgroups, each of which is free of rank $r$---see, for instance, \cite{K:matrices} or \cite{D:stationary}.  
This implies that every rank-$r$ torsion-free abelian group is isomorphic to the inductive limit of an inductive system of the form 

\centerline{
\xymatrix{
\Z^r \ar@{>}^{A_1}[r] & \Z^r \ar@{>}^{A_2}[r] & \Z^r \ar@{>}^{A_3}[r] & \Z^r \ar@{>}^{A_4}[r] & \Z^r \ar@{>}^{A_5}[r] & \cdots ;\\
}
}

let us denote this inductive limit by $\varinjlim (\Z^r; A_i)$.  

We may also suppose that $A_i: \Z^r\to \Z^r$ is an injective endomorphism of $\Z^r$.  
In this case, the endomorphisms $A_i$ induce automorphisms of $\Q^r$, and it is not difficult to verify that the homomorphisms $A_1^{-1}\cdots A_n^{-1} : \Z^r \to \Q^r$ form a compatible family, and the inductive limit is isomorphic to the union of their images: 
\[
\varinjlim (\Z^r; A_i) \cong \bigcup_{n=0}^\infty A_1^{-1}\cdots A_n^{-1} (\Z^r) \subseteq \Q^r.  
\]

The main result of this section describes the set of $p$-adic functionals on such an inductive limit.  
It refers to the endomorphisms $A_i$ from a variety of points of view, which should be clarified.  

Given a basis of $\Z^r$, every endomorphism $A$ of $\Z^r$ can be represented as left multiplication by an $r\times r$ matrix with integer entries.  
Let us also denote this matrix by $A$.  
The endomorphism $A$ induces an endomorphism on $\Z_p^r$, the $p$-adic completion of $\Z^r$; this endomorphism can be represented as left multiplication by the same matrix $A$.  
Then $A$ also induces endomorphisms of $(\Z^r)^*$ and $(\Z_p^r)^*$; viewing these as groups consisting of rows, the induced endomorphisms correspond to right multiplication by the matrix $A$.  
The following is the main result of this section, and will be familiar to readers who are acquainted with the theory of state spaces of dimension groups \cite{G:book}.  
\begin{prop}\label{PROP:limit-functionals}
For each $i\in\N$, let $A_i: \Z^r\to \Z^r$ be an injective endomorphism of $\Z^r$, and consider the group $G = \bigcup_{n=0}^\infty A_1^{-1}\cdots A_n^{-1} (\Z^r)$.  
Then, for any prime $p$, it is true that $G^{*p} = \bigcap_{n=0}^\infty (\Z_p^r)^* A_n\cdots A_1$.  
\end{prop}
\begin{proof}
For $n\geq 0$, let us write $G^{*p}(n) := (\Z_p^r)^* A_n\cdots A_1$ for brevity.  
It is easy to verify that $G^{*p}(n+1)\subseteq G^{*p}(n)$.  

The conclusion of the proposition is that $G^{*p} = \bigcap_{n=0}^\infty G^{*p}(n)$.  
To see this, suppose $v\in \bigcap_{n=0}^\infty G^{*p}(n)$.  
$v$ is clearly a group homomorphism from $G$ to $\Q_p$; thus it is necessary to show that the image of $v$ is contained in $\Z_p$.  

Pick $g\in G$.  
Then $g = A_1^{-1}\cdots A_N^{-1}g_0$ for some $g_0\in\Z^r$.  
Likewise, $v = v_0 A_N\cdots A_1$ for some $v_0\in (\Z_p^r)^*$.  
Then $vg = v_0A_n\cdots A_1 A_1^{-1}\cdots A_N^{-1} g_0 = v_0g_0 \in \Z_p$.  
As $g\in G$ was arbitrary, this proves that $v\in G^{*p}$.  

Now suppose that $v\in (\Z_p^r)^*\backslash \bigcap_{n=0}^\infty G^{*p}(n)$.  
This means in particular that there exists $N\in \N$ such that $v\notin G^{*p}(N)$---in other words, $vA_1^{-1}\cdots A_N^{-1}\in (\Q_p^r)^*\backslash (\Z_p^r)^*$.  
Say that the $i$th entry of $vA_1^{-1}\cdots A_N^{-1}$ is not in $\Z_p$.  
Then, if $e_i$ is the $i$th standard basis element of $\Z^r\subseteq G$, it is true that $A_1^{-1}\cdots A_N^{-1}e_i\in G$, but $vA_1^{-1}\cdots A_N^{-1}e_i\notin \Z_p$.  
Thus $v\notin G^{*p}$.  
\end{proof}

Note that, according to Proposition \ref{PROP:limit-functionals}, $G^{*p}$ is an intersection of a sequence of nested closed subsets of $\Z_p^r$, all of which contain $0$.  
Thus $G^{*p}$ contains at least the zero row; it was proved in a more general setting in \cite[Corollary 3.15]{M:functionals} that $G^{*p}$ contains non-zero functionals if $G$ is not $p$-divisible.  

\section{Relationship to the tensor product}\label{SEC:tensor}

Let $G$ be a torsion-free abelian group.  
The $p$-adic double dual, $G^{*p*p}$, is studied in \cite[Section 3.4]{M:functionals}.  
There is a natural homomorphism from $G$ to $G^{*p*p}$, and the image of an element $g\in G$ is denoted by $\hat{g}$.  

There is a clear relationship between the $p$-adic double dual $G^{*p*p}$ and the tensor product $G\otimes \Z_p$, at least in the case that $G$ has finite rank.  
A key ingredient in the proof of this is the fact, proved in \cite[Theorem 4.3]{M:functionals}, that $G^{*p*p}$ is the pro-$p$ completion \cite{RZ:profinite} of $G$.  
Thus, in particular, $G^{*p*p}$ satisfies the universal property of the pro-$p$ completion, namely, that any group homomorphism from $G$ to a pro-$p$ group factors through $G^{*p*p}$ (and the homomorphism from $G^{*p*p}$ is continuous).  

The other main ingredient is the fact, proved in \cite{F:book}, that $G\otimes \Z_p$ is a direct sum of copies of $\Q_p$ and $\Z_p$, with the total number of summands equal to the rank of $G$.  
The divisible subgroup of $G\otimes \Z_p$ comprises the $\Q_p$ summands; the result here is that the quotient of $G\otimes \Z_p$ by this divisible subgroup is isomorphic to $G^{*p*p}$.  
\begin{thm}\label{THM:tensor}
Let $p$ be a prime and let $G$ be a finite-rank torsion-free abelian group.  
Let $D$ be the divisible subgroup of $G\otimes \Z_p$.  
Then there is a continuous isomorphism between $G^{*p*p}$ and $(G\otimes \Z_p)/D$ that sends $\hat{g}$ to $g\otimes 1 + D$ for all $g\in G$.  
\end{thm}
\begin{proof}
There is a middle linear map from the product $G\times \Z_p$ to $G^{*p*p}$ that sends $(g,\alpha)$ to $\alpha \hat{g}$ for all $g\in G, \alpha\in \Z_p$.  
Hence there exists a group homomorphism $f : G\otimes \Z_p \to G^{*p*p}$ with the property that $f(g\otimes \alpha) = \alpha \hat{g}$.  
$G^{*p*p}$ is reduced---i.e., its divisible subgroup is trivial---so $D$ is contained in the kernel of $f$.  
Thus $f$ induces a group homomorphism $\tilde{f} : (G\otimes \Z_p)/D\to G^{*p*p}$.  
Let us show that $\tilde{f}$ is the required isomorphism.  

$(G\otimes \Z_p)/D$ and $G^{*p*p}$ are both free $\Z_p$-modules of finite $\Z_p$-rank, and $\tilde{f}$ is, in fact, a $\Z_p$-module homomorphism, because $f$ is a $\Z_p$-module homomorphism and $D$ is a $\Z_p$-submodule of $G\otimes \Z_p$.  
Hence $(G\otimes \Z_p)/D$ and $G^{*p*p}$ are both pro-$p$ groups, and moreover the $p$-adic topology agrees with the pro-$p$ topology on both of them.  
(The $p$-adic topology on an abelian group $H$ has as a basis the collection of cosets of subgroups, the quotients of which are $p$-groups; for the pro-$p$ topology these quotients must be \emph{finite} $p$-groups.  
These two notions agree for finite-rank free $\Z_p$-modules.)  

A $\Z_p$-module homomorphism between two finite-rank $\Z_p$-modules is automatically continuous with respect to the $p$-adic topology; hence $\tilde{f}$ is a continuous group homomorphism between two pro-$p$ groups.  

On the other hand, the universal property of the pro-$p$ completion means that there exists a continuous group homomorphism $\eta : G^{*p*p}\to (G\otimes \Z_p)/D$ with the property that $\eta(\hat{g}) = g\otimes 1 + D$ for all $g\in G$.  
But then $\tilde{f}\circ \eta : G^{*p*p}\to G^{*p*p}$ is a continuous group homomorphism that agrees with $id_{G^{*p*p}}$ on the dense subgroup $\{ \hat{g} : g\in G\}$; hence $\tilde{f}\circ \eta$ is the identity on $G^{*p*p}$.  
Similarly, $\eta\circ \tilde{f}: (G\otimes\Z_p)/D \to (G\otimes \Z_p)/D$ is a continuous group homomorphism that agrees with the identity on the subgroup generated by $g\otimes 1 + D$, which is easily seen to be dense.  
\end{proof}

Theorem \ref{THM:tensor} establishes the connection between $p$-adic functionals on $G$ and the Kurosch matrices \cite{K:matrices} of $G$ if $G$ has finite rank.  
Specifically, the Kurosch matrices have entries that express the elements of some basis of $G$ in terms of a $\Q_p$-basis for $D$ and a $\Z_p$-basis for $(G\otimes \Z_p)/D$ (see \cite{F:book} for details).  
The factored-form matrices of \cite{M:functionals} are obtained from the Kurosch matrices by dropping the columns that correspond to $D$ (and taking transposes).  

It is unclear to me if Theorem \ref{THM:tensor} still holds if the hypothesis of finite rank is dropped; the issue is whether or not the maps are still continuous in this case.

\section{The stationary case}\label{SEC:stationary-states}

This section deals with torsion-free abelian groups $G$ that are inductive limits of stationary inductive sequences of finitely-generated free abelian groups.  

\begin{defn}\label{DEF:stationary}
Let $G$ be a torsion-free abelian group.  
Let us say that $G$ is \defemph{stationary} if it is isomorphic to the inductive limit $\varinjlim (\Z^r,A)$ of a stationary inductive sequence of finite-rank free abelian groups in which the homomorphism $A$ is the same at every stage.  
\end{defn}

Objects that are isomorphic to stationary inductive limits in the category of \emph{ordered} abelian groups have been studied before in \cite{H:irrational} and \cite{M:stationary}.  
The groups here need not have any order structure; stationary groups of this type have been studied before in \cite{D:stationary}.  

Suppose that $G = \varinjlim (\Z^r;A)$ for some endomorphism $A$ of $\Z^r$ (which, given a basis of $\Z^r$, can be represented as a matrix in $M_{r\times r}(\Z)$, which will also be denoted by $A$).  
Then we can explicitly describe the $p$-adic functionals on $G$ using basic linear algebra.  
For this purpose it will be useful to assume that the matrix $A$ is non-singular---i.e., its determinant is non-zero---which we may do without loss of generality by arguments from \cite[p. 64]{H:irrational} and \cite[Remark 5.5]{M:stationary}.  
In this case, as in Section \ref{SEC:inductive-limit}, $G$ is isomorphic to $\bigcup_{n=0}^\infty A^{-n}(\Z^r)$.  

\subsection{$p$-divisibility}\label{SUBSEC:p-divisibility}

Recall that, if $n\in\N$, then an element $g$ of a torsion-free abelian group $G$ is called \defemph{$n$-divisible} in $G$ if $g\in nG$, and $G$ is called \defemph{$n$-divisible} if each $g\in G$ is $n$-divisible.  
Let $p$ be a prime.  
Proposition \ref{PROP:p-divisible}, below, gives a characterization of $p$-divisibility for a stationary limit $G$.  
Its proof uses the following two lemmas.  

\begin{lemma}\label{LEM:root-count}
Let $R$ be a discrete valuation ring with uniformizing parameter $\pi$ and absolute value $|\cdot |_\pi$.  
Let $f(x) = x^r + \alpha_1x^{r-1} + \cdots + \alpha_{r-1}x + \alpha_r\in R[x]$ be a monic polynomial (where implicitly $\alpha_0 = 1$).  
Let $K$ denote the splitting field of $f(x)$ over the field of fractions of $R$, and let $|\cdot |_K$ denote its absolute value with respect to some uniformizing parameter.  
Then the number of roots $z$ of $f(x)$, counted with multiplicity, for which $|z|_K = 1$ is equal to $\max_{0\leq i \leq r} \{ i : |\alpha_i|_\pi = 1\}$.  
\end{lemma}
\begin{proof}
Note that, for $\alpha\in R$, $|\alpha|_\pi = |\alpha|_K^l$ for some $l\in\N$.  

Let $z_1,\ldots,z_r$ be the roots of $f(x)$, listed with multiplicity.  
The coefficient $\alpha_j$ of $f(x)$ is the $j$th elementary symmetric polynomial in $z_1,\ldots, z_r$; that is, 
\[
\alpha_j = \sum_{1\leq i_1 < \cdots < i_j \leq r} z_{i_1}\cdots z_{i_j}.  
\]

Note that $|z_i|_\pi \leq 1$ for all $1\leq i\leq r$.  
If exactly $k$ of the roots, counted with multiplicity, satisfy $|z_i|_K = 1$, then suppose without loss of generality that these are the first $k$ roots $z_1,\ldots, z_k$.  
We must show that (1) $|\alpha_k|_\pi = 1$ and (2) $|\alpha_j|_\pi < 1$ for all $j > k$.  

Let us first prove (1): $|\alpha_k|_\pi=1$.  
$\alpha_k$ is a sum of ${r}\choose{k}$ terms.  
One of these terms is $z_1\cdots z_k$, which has absolute value $1$.  
All the other terms have the form $z_{i_1}\cdots z_{i_k}$, where at least one of the indices $i_j$ is greater than $k$; hence $|z_{i_j}|_K < 1$, which implies $|z_{i_1}\cdots z_{i_k}|_K < 1$.  
Thus $|\alpha_k - z_1\cdots z_k|_K < 1$, so by the strong triangle inequality, $1 = |z_1\cdots z_k|_K \leq \max \{ |\alpha_k|_K,|\alpha_k-z_1\cdots z_k|_K\}$, so $|\alpha_k|_K = 1$.  
Then $|\alpha_k|_\pi = |\alpha_k|_K^l = 1$.  

Now let us prove (2).  
For any $j$, $\alpha_j$ is a sum of ${n}\choose{j}$ terms, each of which is a product of $j$ of the roots of $f$.  
If $j > k$, then each of these terms has as a factor at least one root $z_i$ with $|z_i|_K < 1$, so $|\alpha_j|_K \leq \max_{1\leq i_1<\cdots < i_j\leq r} \{ |z_{i_1}\cdots z_{i_j}|_K\} < 1$.  
Thus $|\alpha_j|_\pi = |\alpha_j|_K^l <1$.  
\end{proof}

The following lemma says, in particular, that an $r\times r$ integer matrix has some power with all entries divisible by the prime $p$ if and only if its characteristic polynomial has all coefficients except the leading one divisible by $p$.  
\begin{lemma}\label{LEM:charpoly-coeffs}
Let $R$ be a discrete valuation ring with uniformizing parameter $\pi$, and let $A\in M_{r\times r}(R)$ be a matrix.  
The following are equivalent:  
\begin{enumerate}
\item  There is some positive integer power of $A$ that has all entries in $\pi R$.  
\item  All coefficients of the characteristic polynomial of $A$ except the leading coefficient are in $\pi R$.  
\end{enumerate}
\end{lemma}
\begin{proof}
Let $\chi_A(x) = x^r + \alpha_{1}x^{r-1}+\cdots + \alpha_{r-1}x+\alpha_r$ be the characteristic polynomial of $A$.  
By the Cayley--Hamilton Theorem, $A^r = -\alpha_{1}A^{r-1}-\cdots - \alpha_rI$, which makes it clear that (2) implies (1).  

To see that (1) implies (2), suppose that $A^n$ has all entries in $\pi R$ for some $n\in \N$.  
Factor $\chi_A(x)$ over $K$, the splitting field of $\chi_{A}(x)$ over the field of fractions of $R$: $\chi_A(x) = (x-z_1)\cdots (x-z_r)$.  
Then the characteristic polynomial of $A^n$ is $\chi_{A^n}(x) = (x-z_1^n)\cdots (x-z_r^n)$.  
(To see that this is true, consider the Jordan forms of $A$ and $A^n$.)  
But then, as all entries of $A^n$ are in $\pi R$, the determinant formula $\chi_{A^n}(x) = \det ( xI - A^n)$ implies that $\chi_{A^n}(x)$ has all coefficients except the leading one in $\pi R$.  
By Lemma \ref{LEM:root-count}, this means that all of the roots of $\chi_{A^n}(x)$ have absolute value less than $1$ in the splitting field of $\chi_{A^n}(x)$ over the field of fractions of $R$.  
This field is contained in $K$, and the roots of $\chi_{A^n}(x)$ are the $n$th powers of the roots of $\chi_A(x)$, so the latter also all have absolute value less than $1$.  
Thus, by Lemma \ref{LEM:root-count}, $\chi_A(x)$ has all coefficients except for the leading one in $\pi R$.  
\end{proof}

The following Proposition can be proved using Lemma \ref{LEM:charpoly-coeffs}; it is also a consequence of Theorem \ref{THM:stationary-states}, below.  
\begin{prop}\label{PROP:p-divisible}
Let $p$ be a prime and let $A\in M_{r\times r}(\Z)$ be a non-singular matrix.  
Consider the group $G = \bigcup_{n=0}^\infty A^{-n}(\Z^r)$.  
The following are equivalent:  
\begin{enumerate}
\item  There is some positive integer power of $A$ that has all entries divisible by $p$.  
\item  All coefficients of the characteristic polynomial of $A$ except the leading coefficient are divisible by $p$.  
\item  $G$ is $p$-divisible.  
\end{enumerate}
\end{prop}
\begin{proof}
The equivalence of statements (1) and (2) follows from Lemma \ref{LEM:charpoly-coeffs} by letting $\Z_p$ play the role of $R$ and $p$ play the role of $\pi$.  

To see that (1) and (3) are equivalent, note that $v\in \Q^r$ is in $G$ if and only if there is some $n\in \N$ such that $A^nv\in\Z^r$.  
Then $g\in G$ is $p$-divisible if and only if there is some $n\in\N$ such that $A^ng\in p\Z^r$.  
Thus if $G$ is $p$-divisible then in particular for each $i\leq r$, the standard basis element $e_i$ of $\Z^r\subseteq G$ is $p$-divisible in $G$, so there exists $n_i\in \N$ such that $A^{n_i}e_i\in p\Z^r$.  
Let $N = \max_{i\leq r}\{ n_i\}$; then $A^N$ has all entries divisible by $p$.  
Hence (3) implies (1).  

If (1) is true---say $A^N$ has all entries divisible by $p$---then choose $g\in G$.  
By the definition of $G$, there exists $n\in \N$ such that $A^ng\in \Z^r$; thus $A^{n+N}g\in p\Z^r$, and so $g$ is divisible by $p$ in $G$.  
As $g$ in $G$ was arbitrary, $G$ is $p$-divisible, so (1) implies (3).  
\end{proof}

Note that a variant of Proposition \ref{PROP:p-divisible} holds even if the matrix $A$ is singular; that is, if $G$ is the inductive limit of a stationary system with matrix $A$, then $G$ is $p$-divisible if and only if statement (1) (equivalently, statement (2)) is true of $A$.  
The proof of this fact is only slightly more involved, but requires an alternative presentation of $G$ that is somewhat more opaque than the one that we have used in the non-singular case.  

\subsection{$p$-adic functionals in the stationary case}\label{SUBSEC:stationary-states}

Let $p$ be a prime and let $A\in M_{r\times r}(\Z)$ be a matrix.  
Let $\chi_A(x)$ be the characteristic polynomial of $A$, and let $K$ denote the splitting field of $\chi_A(x)$ over $\Q_p$.  
Then the ring of integers $\mathcal{O}_K$ of $K$ is a discrete valuation ring; let $\pi$ be a uniformizing parameter for it and let $|\cdot |_\pi$ be the absolute value arising from $\pi$.  
All of the roots $z$ of $\chi_A(x)$ are in $\mathcal{O}_K$, and so satisfy $|z|_\pi \leq 1$.  

If $A$ is non-singular, the inductive limit of the stationary system with matrix $A$ is isomorphic to $G = \bigcup_{n=0}^\infty A^{-n}(\Z^r)$.  
By Proposition \ref{PROP:p-divisible}, this is $p$-divisible if and only if every coefficient of $\chi_A(x)$ except the leading one is divisible by $p$.  
By Lemma \ref{LEM:root-count}, this is equivalent to saying that all roots $z$ of $\chi_A(x)$ satisfy $|z|_\pi < 1$.  
So $G$ is not $p$-divisible if and only if $\chi_A(x)$ has at least one root $z\in K$ with $|z|_\pi = 1$.  
And by \cite[Corollary 3.15]{M:functionals}, $G$ has non-zero $p$-adic functionals if and only if $G$ is not $p$-divisible; therefore the number of roots $z$ of $\chi_A(x)$ with $|z|_\pi = 1$ is of particular interest.  

\begin{defn}\label{DEF:max-roots}
Let $R$ be a discrete valuation ring with field of fractions $Q$.  
Let $f\in R[x]$ be a monic polynomial of degree $r\in\N$.  
Let $K$ denote the splitting field of $f$ over $Q$, and let $|\cdot |_\pi$ denote the absolute value on $K$ with respect to some uniformizing parameter $\pi$.  
Let $z_1,\ldots,z_r$ denote the roots of $f$ in $K$ listed with multiplicity.  
Then define the \defemph{unit root factor} of $f$ to be the product 
\begin{align*}
f^1(x) := \prod_{\{ i : |z_i|_\pi = 1\}} (x-z_i).  
\end{align*}
Define the \defemph{ideal root factor} of $f$ to be the product 
\begin{align*}
f^0(x) := \prod_{\{ i : |z_i|_\pi < 1\}} (x-z_i),  
\end{align*}
so that $f(x) = f^1(x)f^0(x)$.  
\end{defn}
If $R = \Z_p$ in Definition \ref{DEF:max-roots}, then it is a consequence of the Newton Polygon Theorem \cite[Theorem 6.3.1]{C:local-fields} that $f^1(x) \in \Q_p[x]$, and hence $f^0(x)\in\Q_p[x]$ as well.  

The reason for introducing Definition \ref{DEF:max-roots} here is that the unit root factor $\chi_A^1(x)$ of $\chi_A(x)$ gives information about the $p$-adic functionals on $G$.  
$\chi_A^1(x)$ has an associated invariant subspace for $A$; namely, the null space of $\chi^1_A(A)$.  
In fact, there are two invariant subspaces: on the right, there is the subspace $\{ v\in \Q_p^r : \chi^1_A(A)v = 0\}$, and on the left there is the subspace $\{ w\in (\Q_p^r)^* : w\chi^1_A(A) = 0\}$, which are invariant under the left and right actions of $A$, respectively.  
It is the latter of these two subspaces that furnishes the $p$-adic functionals on $G$.  
Specifically, the $p$-adic functionals on $G$ are exactly the elements of this subspace that have entries in $\Z_p$---this is the content of Theorem \ref{THM:stationary-states}, below.  
The proof of this theorem uses Lemmas \ref{LEM:vector-space-properties} and \ref{LEM:module-properties}, which are expressed in terms of the following definition.  

\begin{defn}\label{DEF:module-decomposition}
Let $R$ be a discrete valuation ring and let $M\cong R^r$ be a free $R$-module of finite $R$-rank.  
Let $A: M\to M$ be an $R$-module homomorphism with characteristic polynomial $\chi_A(x)$.  
Then define the \defemph{unit submodule of $M$} associated to $A$ to be 
\begin{align*}
M^1(A) & := \{ v\in M : \chi^1(A)v = 0\},
\end{align*}
and the \defemph{ideal submodule of $M$} associated to $A$ to be 
\begin{align*}
M^0(A) & := \{ v\in M : \chi^0(A)v = 0\}.
\end{align*}

Let $Q$ be the field of fractions of $R$; then $M\subseteq V = Q^r$.  
Define the \defemph{unit subspace of $V$} associated to $A$ to be 
\begin{align*}
V^1(A) & := \{ v\in V : \chi^1(A)v = 0\},
\end{align*}
and the \defemph{ideal subspace of $V$} associated to $A$ to be 
\begin{align*}
V^0(A) & := \{ v\in V : \chi^0(A)v = 0\}.
\end{align*}
\end{defn}

\begin{lemma}\label{LEM:vector-space-properties}
Let $R$ be a discrete valuation ring with uniformizing parameter $\pi$ and let $Q$ denote the field of fractions of $R$.  
Let $M\cong R^r$ be a free $R$-module of finite $R$-rank, let $A: M\to M$ be an $R$-module homomorphism, and let $V^1(A)$ and $V^0(A)$ be the unit and ideal subspaces, respectively, of $Q^r$ associated to $A$.  
Then 
\begin{enumerate}
\item  $Q^r$ is equal to the internal direct sum of $V^1(A)$ and $V^0(A)$, which are $A$-invariant;
\item  $A|_{V^1(A)}$ is an isometry of $Q$-vector spaces; and 
\item  for $n\geq 0$, there exists $k\in\N$ such that $\| A^k(v) \|_\pi < \| v\|_\pi$ for all $v\in Q^r$.  
\end{enumerate}
\end{lemma}
\begin{proof}
Statement (1) follows from basic linear algebra---e.g., by considering the Jordan Normal Form of $A$.  

The characteristic polynomial of $A|_{V^1(A)}$ is $\chi_A^1(x)$, which certainly does not have $0$ as a root, because $|0|_\pi = 0 < 1$.  
Therefore $A|_{V^1}$ is non-singular, and hence invertible on $V^1$.  
Moreover, the operator norm of $A|_{V^1(A)}^{-1}$ is $\inf \{ c \geq 0 : \| A|_{V^1(A)}^{-1} (v) \|_\pi \leq c \| v\|_\pi $ for all $ v\in V^1\}$, which equals the maximum norm of any eigenvalue of $A|_{V^1(A)}^{-1}$---i.e., $1$.  
Thus $A|_{V^1(A)}$ is an isometry, proving statement (2).  

The characteristic polynomial of $A|_{V^0(A)}$ is $\chi_A^0(x)$.  
By Lemma \ref{LEM:root-count}, all of its coefficients except the leading one lie in $\pi R$.  
By Lemma \ref{LEM:charpoly-coeffs}, this implies that, if $B$ is a matrix representation of $A|_{V^0(A)}$ with respect to some basis of $V^0(A)$, then some positive integer power of $B$ has all entries in $\pi R$, which proves statement (3).  
\end{proof}

\begin{lemma}\label{LEM:module-properties}
Let $R$ be a discrete valuation ring with uniformizing parameter $\pi$ and let $M\cong R^r$ be a free $R$-module of finite $R$-rank.  
Let $A: M\to M$ be an $R$-module homomorphism and let $M^1(A)$ and $M^0(A)$ be the unit and ideal submodules, respectively, of $M$ associated to $A$.  
Then 
\begin{enumerate}
\item  $M$ is equal to the internal direct sum of $M^1(A)$ and $M^0(A)$, which are $A$-invariant;
\item  $A|_{M^1(A)}$ is an isometric isomorphism of $R$-modules; and 
\item  for all $n\geq 0$, there exists $k\in\N$ such that $A^k(v)\in \pi^nM$ for all $v\in M$.  
\end{enumerate}
\end{lemma}

\begin{proof}
In the terminology of Lemma \ref{LEM:vector-space-properties}, $M^1(A) = Q\cap V^1(A)$ and $M^0(A) = M\cap V^0(A)$, so statements (2) and (3) follow immediately from the corresponding statements in Lemma \ref{LEM:vector-space-properties}.  

To prove statement (1), note that every $v\in M$ can be written uniquely as a sum $v = w^1 + w^0$, with $w^1\in V^1(A)$ and $w^0\in V^0(A)$, so we must show that $w^1\in M^1(A)$ and $w^0\in M^0(A)$.  
As $v\in M$, it will suffice to show that $w^1\in M^1(A)$.  
By statement (3), there exists $k\in \N$ such that $A^k(w^0)\in M^0(A)$.  
But $A$ is an endomorphism of $M$, so $A^k(v) = A^k(w^1) + A^k(w^0) \in M$, which implies that $A^k(w^1)\in M^1(A)$.  
But by statement (2), $A|_{M^1(A)}$ is invertible; therefore $w^1\in M^1(A)$, which proves statement (1).  
\end{proof}

\begin{thm}\label{THM:stationary-states}
Let $p$ be a prime, let $A\in M_{r\times r}(\Z)$ be a non-singular matrix, let $\chi_A(x)$ be the characteristic polynomial of $A$, let $\chi_A^1(x)$ be the unit root factor of $\chi_A(x)$, and let $G = \bigcup_{n=0}^\infty A^{-n}(\Z^r)$.  
Then $G^{*p} = \{ w\in (\Z_p^r)^* : w\chi_A^1(A) = 0\}$.  
\end{thm}
\begin{proof}
Let $M^1$ and $M^0$ denote the unit and ideal submodules ,respectively, of $(\Z_p^r)^*$ associated to $A$.  
Then the theorem states that $G^{*p} = M^1$.  

By Proposition \ref{PROP:limit-functionals}, $G^{*p} = \bigcap_{n=0}^\infty (\Z_p^r)^* A^n$.  
By part (2) of Lemma \ref{LEM:module-properties}, $A|_{M^1}$ is an isometric isomorphism, from which it easily follows that $M^1\subseteq \bigcap_{n=0}^\infty (\Z_p^r)^* A^n$.  

Now choose $v\in \bigcap_{n=0}^\infty (\Z_p^r)^*A^n$.  
By part (1) of Lemma \ref{LEM:module-properties}, we can write $v = v^1 + v^0$ with $v^1\in M^1$ and $v^0\in M^0$; let us suppose for a contradiction that $v^0\neq 0$.  
Then $\| v^0\|_p > 0$---say $\| v^0\|_p = p^{-N}$.  
But by part (3) of Lemma \ref{LEM:module-properties}, we can find $k\in\N$ such that $wA^k\in p^{N+1}(\Z_p^r)^*$ for all $w\in (\Z_p^r)^*$---i.e., $\|wA^k\|_p \leq p^{-N-1}$.  
This implies that $v^0\notin (\Z_p^r)^*A^k$, which is a contradiction, as $v, v^1\in (\Z_p^r)^*A^k$.  
\end{proof}

The divisibility pseudometric $d_p$ on $G$ was introduced in \cite[Definition 3.2]{M:functionals}; it is defined by $d_p(g,h) = p^{-n}$, where $n$ is maximal with the property that $g-h$ is $p^n$-divisible in $G$, or $d_p(g,h) = 0$ if there is no such maximal $n$.  
The following proposition clarifies the role of the right unit subspace: specifically, the magnitude of the projection of a group element $g$ onto the right unit subspace is equal to $d_p(g,0)$.  
\begin{prop}\label{PROP:norm-decomposition}
Let $p$ be a prime, let $A\in M_{r\times r}(\Z)$ be a non-singular matrix, let $\chi_A(x)$ be the characteristic polynomial of $A$, and let $G = \bigcup_{n=0}^\infty A^{-n}(\Z^r)$.  
Every $g\in G$ can be written $g = g^1 + g^0$ with $g^1$ and $g^0$ in the unit and ideal subspaces of $\Q_p^r$ associated to $A$, and moreover $d_p(g,0) = \| g^1\|_p$, where $d_p$ is the divisibility pseudometric on $G$.  
\end{prop}
\begin{proof}
Note that the statement that $g\in p^kG$ means there exists $n\in \N$ such that $A^{n}g\in p^k\Z^r$---in other words, $\| A^{n}g\|_p \leq p^{-k}$.  
There exists $N\in \N$ such that $\| A^{m}g^0\|_p < p^{-k}$ for all $m\geq N$.  
Thus if $n\geq N$ then 
\begin{align*}
\| A^{n} g\|_p & \leq \max \{ \| A^{n}g^1\|_p,\| A^{n}g^0\|_p\} \\ 
 & \leq \max \{ \| A^{n}g^1\|_p,p^{-k}\} = \max \{ \| g^1\|_p,p^{-k}\},  
\end{align*}
where the last equality comes from part (2) of Lemma \ref{LEM:module-properties}.  
So if $\| g^1\|_p \leq p^{-k}$, then $d_p(g,0)\leq p^{-k}$.  
This shows that $d_p(g,0)\leq \| g^1\|_p$.  

On the other hand, if $d_p(g,0) \leq p^{-k}$, then the above argument shows that $\| A^{n}g\|_p \leq p^{-k}$, and hence 
\begin{align*}
\| g^1\|_p & = \| A^{n}g^1\|_p \\ 
 & \leq \max \{ \| A^{n}g\|_p, \|-A^{n}g^0\|_p\} \\ 
 & \leq \max \{ p^{-k},p^{-k}\} = p^{-k},
\end{align*}
where the first inequality is the strong triangle inequality, and the second inequality holds if $n$ is large enough that $\| A^{n}g^0\|_p \leq p^{-k}$.  
Thus $\|g^1\|_p \leq d_p(g,0)$.  
\end{proof}
Note that Proposition \ref{PROP:norm-decomposition} implies that $\| g^1\|_p\leq 1$ for all $g\in G$.  

\begin{example}\label{EX:dugas}
Let us illustrate the usefulness of the theory of $p$-adic functionals with an example.  
In \cite{D:stationary}, Dugas gave an example of a torsion-free abelian group that is a stationary limit of a matrix with irreducible characteristic polynomial, but is not strongly indecomposable.  
Specifically, consider the algebraic number $\lambda = \sqrt{2} + i\in \C$.  
Dugas showed that the group $G = \mathcal{O}_\lambda \big[ \frac{1}{\lambda}\big]$ (see Definition \ref{DEF:ring-extension} for the notation) is not strongly indecomposable.  
This fact manifests itself clearly in the factored form \cite[Definition 5.7]{M:functionals} (a variant of the Kurosch matrices \cite{K:matrices}) of $G$, which is computed from the $p$-adic functionals.  

The minimal polynomial of $\lambda$ is $\chi(x) = x^4-2x^2+9$.  
Let $A$ denote the companion matrix of $\chi(x)$:
\[
A = \left[ \begin{array}{rrrr} 0 & 0 & 0 & -9 \\ 
1 & 0 & 0 & 0 \\
0 & 1 & 0 & 2 \\
0 & 0 & 1 & 0 \end{array}\right],
\]
and consider the group $G = \bigcup_{n=0}^\infty A^{-n}(\Z^4)$.  
$\chi(x)$ is the characteristic polynomial of $A$, and by Lemma \ref{LEM:root-count} it has exactly two roots---let us call them $\lambda_1$ and $\lambda_2$---that satisfy $|\lambda_i|_\pi < 1$, where $|\cdot |_\pi$ is the extension of the $3$-adic absolute value to the splitting field of $\chi(x)$.  
Then, for $i = 1,2$, $(\lambda_i^3-2\lambda_i, \lambda_i^2-2,\lambda_i,1)^t$, is a right eigenvector of $A$ associated to $\lambda_i$.  

If $\mu$ is any root of $\chi(x)$, then $\mu^2$ is a root of $x^2 - 2x + 9$, which has exactly one root $\nu$ with $|\nu |_\pi < 1$.  
Hence $\lambda_1^2 = \lambda_2^2 = \nu$, and moreover this common value lies in $\Q_3$, and hence $\Z_3$, by the Newton Polygon Theorem.  
Let $\alpha = \lambda_i^2 - 2\in \Z_3$.  
Then, for $i = 1,2$, $(\lambda_i\alpha, \alpha,\lambda_i,1)^t$, is a right eigenvector of $A$ associated to $\lambda_i$.  

The $3$-adic functionals on $G$ are exactly the rows in $(\Z_3^4)^*$ that have a product of $0$ with both of these eigenvectors.  
Thus $G^{*3}$ is spanned by $(0,1,0,-\alpha)$ and $(1,0,-\alpha,0)$.  
For all other primes $p$, $G^{*p}$ is all of $(\Z_p^4)^*$.  
But $G = \{ v\in \Q^4 : $ for all primes $p, wv \in\Z_p $ for all $w\in G^{*p}\}$.  
Inspection of the rows in $G^{*p}$ reveals that, if $(a,b,c,d)^t\in G$, then $(a,0,c,0)^t,(0,b,0,d)^t\in G$; therefore $G$ is the direct sum of $G\cap \Q\oplus 0 \oplus \Q \oplus 0$ and $G\cap 0\oplus \Q \oplus 0 \oplus \Q$.  
\end{example}

\section{Isomorphism and quasi-isomorphism}\label{SEC:isomorphism}

In \cite{D:stationary}, Dugas gave an explicit description of what a finite-rank torsion-free abelian group $G$ must look like, up to quasi-isomorphism (Definition \ref{DEF:quasi-isomorphic}), if it is quasi-isomorphic to a stationary group.  
This result can be strengthened using results from \cite{M:stationary}, which imply in particular that if $G$ is quasi-isomorphic to a stationary inductive limit, then $G$ is isomorphic to a stationary inductive limit.  
\cite{M:stationary} deals with stationary inductive limits in the category of ordered groups, and much of the content of that work has to do with the order structure.  
Therefore a drastically simplified version, which deals only with the group structure and ignores the order structure, is presented here.  

The result of Dugas is expressed in terms of the following two definitions.  

\begin{defn}\label{DEF:quasi-isomorphic}
Let $G$ and $H$ be finite-rank torsion-free abelian groups.  
$G$ and $H$ are called \defemph{quasi-isomorphic} if there exist $n\in \N$ and homomorphisms $\alpha : G\to H$ and $\beta : H\to G$ such that $\alpha\circ \beta = n\cdot id_H$ and $\beta\circ\alpha = n\cdot id_G$.  
\end{defn}

\begin{defn}\label{DEF:ring-extension}
Let $0\neq \lambda\in \C$ be an algebraic integer and let $\mathcal{O}_\lambda$ denote the ring of integers of the field extension $\Q [\lambda]$ of $\Q$ by $\lambda$.  
Then define $L_\lambda$ to be the additive group of 
\[
\mathcal{O}_\lambda \big[\frac{1}{\lambda}\big] = \big\{ \frac{r}{\lambda^n} : r\in \mathcal{O}_\lambda, n\in\N\big\};
\]
that is, $L_\lambda$ is the additive group of the ring extension of $\mathcal{O}_\lambda$ by $1/\lambda$.  
\end{defn}

The following is the main result of \cite{D:stationary}, by Dugas.  
\begin{thm}\cite[Theorem 1]{D:stationary}\label{THM:dugas}
Let $G$ be a finite-rank torsion-free abelian group.  
The following are equivalent.  
\begin{enumerate}
\item  $G$ is quasi-isomorphic to the inductive limit of a stationary inductive system of finite-rank free abelian groups.  
\item  There exists a finite set $\Sigma$ of algebraic integers in $\C$ such that $G$ is quasi-isomorphic to a finite direct sum of copies of $L_\lambda$, $\lambda\in \Sigma$.  
\end{enumerate}
\end{thm}

The following results from \cite{M:stationary} make it possible to replace ``quasi-isomorphic'' in statement (1) of Theorem \ref{THM:dugas} with ``isomorphic.''  
\begin{lemma}\cite[Corollary 6.10]{M:stationary}\label{LEM:matrix-power-difference}
Let $B$ be a square integer matrix and let $m\geq 2$ be an integer.
Then there exist integers $k > l\geq 0$ such that every entry of the matrix $B^k - B^l$ is divisible by $m$.
\end{lemma}

\begin{lemma}\cite[Proposition 5.4]{M:stationary}\label{LEM:increasing-subgroup}
Let $G$ be a torsion-free abelian group.  
Then $G$ is stationary if and only there exist a finitely-generated free abelian subgroup $F\subseteq G$ and an automorphism $\alpha: G\to G$ such that $F\subseteq \alpha(F)$ and $G = \bigcup_{n=0}^\infty \alpha^{n}(F)$.  
\end{lemma}

The following lemma is a consequence of \cite[Lemma 6.12]{M:stationary}.  
\begin{lemma}\label{LEM:add-brace}
Let $G$ be a stationary torsion-free abelian group and choose an arbitrary $z\in G\otimes \Q$.  
Then the subgroup of $G\otimes \Q$ generated by $G$ and $z$ is stationary.  
\end{lemma}
\begin{proof}
By Lemma \ref{LEM:increasing-subgroup}, there exist a finitely-generated free abelian subgroup $F\subseteq G$ and an automorphism $\alpha : G\to G$ such that $F\subseteq \alpha(F)$ and $G = \bigcup_{n=0}^\infty \alpha^{n}(F)$.  
Then $\alpha$ induces an automorphism of $G\otimes \Q$; let us denote this automorphism by $A$.  

Let $G'$ denote the subgroup of $G\otimes \Q$ generated by $G$ and $z$, and let $F' = \langle F \cup \{ z\}\rangle$.  
Then $F'$ is free abelian and finitely generated.  
Let us show that there exists some power $A^k$ of $A$ that restricts to an automorphism of $G'$ with the property that $F'\subseteq A^k(F')$ and $G' = \bigcup_{n=0}^\infty A^{nk}(F')$; by Lemma \ref{LEM:increasing-subgroup}, this will suffice to show that $G'$ is stationary.  

To show that $A^k$ restricts to an automorphism of $G'$, it is enough to show that $A^k(z)-z\in G$---this means that $A^k$ restricts to an endomorphism of $G'$, and then the fact that $A^k$ is an automorphism of $G$ implies that $A^{-k}(z)-z\in G$, which implies that $A^{-k}$ also restricts to an endomorphism of $G'$.  
Choose a basis $g_1,\ldots, g_r$ of $F$.  
$F\subseteq \alpha(F)$, so it is possible to write each $g_i$ as an integer combination of the elements $\alpha(g_j)$.  
Let $B$ be a transition matrix for $A$ with respect to this basis; that is, the entries $B_{ij}$ of $B$ satisfy 
\[
g_i = \sum_{j=1}^r B_{ji} \alpha(g_j)
\]
(note the reversal of the indices).  

There exists $m\in\N$ such that $mz = g\in F$.  
Write $g$ as an integer combination of the basis for $F$: $g = v_1g_1+\cdots + v_rg_r$, and let $v = (v_1,\ldots, v_r)^t\in \Z^r$.  

By Lemma \ref{LEM:matrix-power-difference}, there exist integers $l_1 > l_2 \geq 0$ such that $B^{l_1}-B^{l_2}$ has all entries divisible by $m$.  
The entries of $B^{l_1}v$ express $g$ as an integer combination of $\alpha^{l_1}(g_1),\ldots , \alpha^{l_1}(g_r)$, while the entries of $B^{l_2}v$ express $\alpha^{l_1-l_2}(g)$ as an integer combination of the same elements.  
The fact that $B^{l_1}-B^{l_2}$ has all entries divisible by $m$ implies in particular that $(B^{l_1}-B^{l_2})v$ has all entries divisible by $m$, which implies that $g-\alpha^{l_1-l_2}(g)\in m \alpha^{l_1}(F)\subseteq mG$.  
Thus $z-A^{l_1-l_2}(z) = \frac{1}{m}(g-\alpha^{l_1-l_2}(g)) \in G$, so $A^{l_1-l_2}$ restricts to an automorphism of $G'$.  

Let $k$ be the least common multiple of $l_1$ and $l_1-l_2$.  
Then $F\subseteq A^k(F)$, and the above argument shows that $z\in A^{l_1}(F) + A^{l_1-l_2}(z)\Z\subseteq A^{k}(F')$.  
Thus $F'\subseteq A^k(F')$.  
Moreover, as $A^k$ is an automorphism of $G'$, it is true that $\bigcup_{n=0}^\infty A^{nk}(F')\subseteq G'$; conversely $G = \bigcup_{n=0}^\infty \alpha^n(F) \subseteq \bigcup_{n=0}^\infty A^{nk}(F')$, and $z\in F'\subseteq \bigcup_{n=0}^\infty A^{nk}(F')$.  
Therefore $G' = \bigcup_{n=0}^\infty A^{nk}(F')$, and hence $G'$ is stationary.  
\end{proof}

We now have the necessary ingredients to prove the following theorem, which says that the class of stationary torsion-free abelian groups is closed under quasi-isomorphism.  
\begin{thm}\label{THM:quasi-isomorphism}
Let $G$ be a finite-rank torsion-free abelian group.  
If $G$ is quasi-isomorphic to a stationary torsion-free abelian group, then $G$ is stationary.  
\end{thm}
\begin{proof}
Suppose that $G$ is quasi-isomorphic to the stationary torsion-free abelian group $H$; say $\alpha : G\to H$ and $\beta : H\to G$ are group homomorphisms satisfying $\alpha\circ \beta = n\cdot id_H$ and $\beta\circ\alpha = n\cdot id_G$.  
$n\cdot id_G$ is injective, so $\beta$ is injective.  
The image $\beta(H)$ in $G$ is isomorphic to $H$, and hence stationary; moreover, as $n\cdot id_G(G) = nG$, it must be true that $nG \subseteq \beta (H)$.  

As $G$ has finite rank, the quotient $G/nG$ is finite, so $G/\beta(H)$ is finite too; let $z_1,\ldots, z_m\in G$ be a complete set of coset representatives for $G/\beta(H)$.  
Every element $g\in G$ has the property that $ng\in \beta(H)$, so $G$ embeds naturally in $\beta(H)\otimes \Q$, and in fact $G$ is generated by $\beta(H)$ and $z_1,\ldots, z_m$.  
Then, using induction and Lemma \ref{LEM:add-brace}, it follows that $G$ is stationary.  
\end{proof}

\bibliographystyle{abbrv}
%\bibliography{finite-rank}

\begin{thebibliography}{1}

\bibitem{C:local-fields}
J.~W.~S. Cassels.
\newblock {\em Local fields}, volume~3 of {\em London Mathematical Society
  Student Texts}.
\newblock Cambridge University Press, Cambridge, 1986.

\bibitem{D:stationary}
M.~Dugas.
\newblock Torsion-free abelian groups defined by an integral matrix.
\newblock {\em Int. J. Algebra}, 6(1-4):85--99, 2012.

\bibitem{F:book}
L.~Fuchs.
\newblock {\em Infinite abelian groups. {V}ol. {II}}.
\newblock Academic Press, New York-London, 1973.
\newblock Pure and Applied Mathematics. Vol. 36-II.

\bibitem{G:book}
K.~R. Goodearl.
\newblock {\em Partially ordered abelian groups with interpolation}, volume~20
  of {\em Mathematical Surveys and Monographs}.
\newblock American Mathematical Society, Providence, RI, 1986.

\bibitem{H:irrational}
D.~Handelman.
\newblock Positive matrices and dimension groups affiliated to {$C^{\ast}
  $}-algebras and topological {M}arkov chains.
\newblock {\em J. Operator Theory}, 6(1):55--74, 1981.

\bibitem{K:matrices}
A.~Kurosch.
\newblock Primitive torsionsfreie {A}belsche {G}ruppen vom endlichen {R}ange.
\newblock {\em Ann. of Math. (2)}, 38(1):175--203, 1937.

\bibitem{M:functionals}
G.~R. Maloney.
\newblock $p$-adic functionals on torsion-free abelian groups.
\newblock 2016.

\bibitem{M:stationary}
G.~R. Maloney.
\newblock Simple dimension groups that are isomorphic to stationary inductive
  limits.
\newblock 2016.

\bibitem{RZ:profinite}
L.~Ribes and P.~Zalesskii.
\newblock {\em Profinite groups}, volume~40 of {\em Ergebnisse der Mathematik
  und ihrer Grenzgebiete. 3. Folge. A Series of Modern Surveys in Mathematics
  [Results in Mathematics and Related Areas. 3rd Series. A Series of Modern
  Surveys in Mathematics]}.
\newblock Springer-Verlag, Berlin, 2000.

\end{thebibliography}

\end{document}